\documentclass[11pt,a4paper,reqno]{amsart}

\usepackage{amsmath,amssymb,amsfonts,amsthm,mathtools}
\usepackage{enumerate}
\usepackage[colorlinks=true,citecolor=black,pagebackref=false]{hyperref}

\numberwithin{equation}{section}
\allowdisplaybreaks

\usepackage{ulem}
\newcommand{\R}{\mathbb R}
\newcommand{\N}{\mathbb N}
\newcommand{\E}{\mathbb E}
\newcommand{\supp}{\mathrm{supp}}

\newcommand{\set}[1]{\left\{#1\right\}}

\makeatletter
\@namedef{subjclassname@2020}{\textup{2020} Mathematics Subject Classification}
\makeatother
\newtheorem{thm}{Theorem}[section]
\newtheorem{prop}[thm]{Proposition}
\newtheorem{lem}[thm]{Lemma}

\newtheorem{rem}[thm]{Remark}
\newtheorem{ques}[thm]{Question}

\title[Almost everywhere convergence of Mock Fourier series]
{Almost everywhere convergence of mock Fourier series for the middle-fourth Cantor measure}

\author{Min-Wei Tang}

 \address[Min-Wei Tang]{School of Mathematics and Statistics, Hunan Normal University, Changsha, 410081, P.~R.~China}
\email{}

\author{Zhi-Yi Wu}
\address[Zhi-Yi Wu]{School of Mathematics and Information Science, Guangzhou University, Guangzhou, 510006, P.~R.~China}
\email{zhiyi\_wu2021@163.com}



\begin{document}
\begin{abstract}
In 1998, Jorgensen and Pedersen constructed the first example of a
singular continuous spectral measure. Precisely, they proved that the
self-similar measure generated by the iterated function system
$\{\frac{x-1}{4},\frac{x+1}{4}\}$ with equal weights, denoted by
$\mu_{1/4}$, is a spectral measure with a spectrum $\Lambda_4$, called the
canonical spectrum,\[
    \Lambda_4
      :=
      \set{
        \sum_{j=0}^{m-1}\varepsilon_j4^j:
        m\ge 1,\ \varepsilon_j\in\{0,1\}
      }.
\] For $f\in L^1(\mu_{1/4})$, let
$S_n f$ be the $n$-th partial sum of its Mock Fourier series with
respect to $\Lambda_4$.  We prove that the associated maximal operator
$S^{\ast}f=\sup_n|S_n f|$ is of weak type $(1,1)$.  Consequently,
$S_n f\to f$ $\mu_{1/4}$-almost everywhere on $\supp(\mu_{1/4})$. This solves a
long-standing open problem of Strichartz \cite[p.~341]{Str06}.
\end{abstract}

\maketitle

\section{Introduction}
For $\rho\in(0,1)$, the {\it Bernoulli convolution} $\mu_\rho$ is the probability measure on $\R$ that is the  distribution of the random variable \(\sum_{j=1}^\infty \xi_j \rho^j\), where \(\xi_i \) are i.i.d. random variables assuming values $1$ and $-1$ with equal probability. Equivalently, $\mu_\rho$ is the unique Borel probability measure satisfying
\[
\mu_\rho(\cdot)=\frac12\,\mu_\rho(\tau_1^{-1}(\cdot))+\frac12\,\mu_\rho(\tau_2^{-1}(\cdot)),
\]
where $\tau_1(x)=\rho(x-1)$ and $\tau_2(x)=\rho(x+1)$ (see \cite{Hut81}). This measure is supported on a compact $T_\rho$ as follows:
 \begin{equation}\label{eq:expand}
 T_\rho=\left\{\sum_{j=1}^\infty c_j\rho^j:~c_j\in\{-1,1\}~\text{for all}~j\right\}.
 \end{equation}

 The study of Bernoulli convolutions dates back to the 1930s and was
revived in the 1980s, revealing surprising connections with harmonic
analysis, number theory, dynamical systems, and fractal geometry;
see the survey \cite{PSS00}.  In the 21st century the subject has
witnessed a series of spectacular advances, including the breakthrough works of
Shmerkin~\cite{Shm14} and Varj\'{u}~\cite{Var19a,Var19b} on the absolute continuity and
dimension of Bernoulli convolutions.

In 1998, Jorgensen and Pedersen~\cite{JP98} discovered the first example of a non-atomic singular spectral measure, i.e., the Bernoulli convolution $\mu_\rho$ for $\rho=1/(2k)$, where $k\ge 2$ is an integer. Here, a Borel probability measure $\nu$ on $\R^d$ is called a {\it spectral measure} if there exists a countable set $\Lambda\subseteq\R^d$ such that $\{e^{2\pi i\langle\lambda,x\rangle}:\lambda\in\Lambda\}$ is an orthonormal basis for $L^2(\nu)$. Such a set $\Lambda$ is called a {\it spectrum} of $\nu$.

Since the celebrated  work of Jorgensen and Pedersen, the spectrality of Bernoulli convolutions has been studied extensively. When $1/\rho\in\N$, Dutkay and Jorgensen~\cite{DJ07} proved that $\mu_\rho$ is a spectral measure if and only if $\rho=1/(2k)$ for some integer $k\ge 1$. Hu and Lau~\cite{HL08} obtained a necessary condition: any spectral measure \(\mu_\rho\) must satisfy \(\rho=(p/q)^{1/r}\) for some positive integers \(p,q,r\) with \(p\) odd and \(q\) even. A full classification of the spectrality of Bernoulli convolutions for all $\rho\in (0,1)$ was ultimately obtained by Dai in \cite{Dai12}, showing that the only spectral Bernoulli convolution are those with $\rho=1/(2k)$ for $k\ge 1$..

\begin{thm}[{\cite{Dai12}}]\label{thm:dai}
Let $\rho\in (0,1)$. Then the Bernoulli convolution $\mu_\rho$ is a spectral measure if and only if $\rho=1/(2k)$ for some integer $k\ge 1$.
\end{thm}

For $k=1$, $\mu_{1/2}=\frac{1}{2}\mathcal{L}|_{[-1,1]}$, where $\mathcal{L}$ denotes the Lebesgue measure, and it is well-known that $\Lambda_2:=\frac{1}{2}\mathbb{Z}$ is the unique spectrum of $\mu_{1/2}$, up to translations. Jorgensen and Pedersen~\cite{JP98} proved that for $k\geq2$, the set
\[
\Lambda_{2k}= \Biggl\{ \sum_{j=0}^{m-1}c_j (2k)^j : m\ge 1,\ c_j\in\{0,k/2\}~\text{for}~ 2\leq j\leq m-1\Biggr\}
\]
is a spectrum of $\mu_{1/(2k)}$, referred to as the {\it canonical spectrum}.

Following Strichartz~\cite{Str00,Str06}, for each $f\in L^1(\mu_{1/(2k)})$, we define the Fourier coefficients
\[
\widehat f(\lambda)=\int f(x)\,e^{-2\pi i\lambda x}\,\mathrm d\mu_{1/(2k)}(x),\quad \lambda\in\Lambda_{2k}.
\]
The trigonometric series $\sum_{\lambda\in\Lambda_{2k}}\widehat f(\lambda)e^{2\pi i\lambda x}$ is called the {\it Mock Fourier series} of $f$ with respect to (w.r.t.) the spectrum $\Lambda_{2k}$.

For $k=1$, the spectrum $\Lambda_2$ reduces the
trigonometric series to the classical Fourier series on the interval $[-1,1]$.
In this case, for $1< p<\infty$, both $L^p$ convergence and almost everywhere convergence for $f\in L^p(\mu_{1/2})$ are classical theorems in \cite{Car66,Hunt68}.

For $k\geq2$ and $n\ge 1$, let
\[
\Lambda_{2k}^{(n)}=\Biggl\{\sum_{j=0}^{n-1} c_j(2k)^j : c_j\in\{0,k/2\}\Biggr\}.
\]
The $n$-th partial sum is given by
\[
S_n f(x)=\sum_{\lambda\in\Lambda_{2k}^{(n)}} \widehat f(\lambda)\, e^{2\pi i\lambda x}.
\]

Similar to the classical situation, for $k\geq2$, we have the following question:
\begin{ques}
In what sense does $S_n f$ converge to $f$ as $n\rightarrow+\infty$?
\end{ques}

Since $\mu_{1/(2k)}$ is a spectral measure, it follows immediately that $S_n f\to f$ in $L^2$ for all $f\in L^2(\mu_{1/(2k)})$. Later, Strichartz \cite{Str06} proved the stronger result that for every $1\le p<\infty$,
\[
S_n f\to f\quad\text{in }L^p\text{ for all }f\in L^p(\mu_{1/(2k)}).
\]
Moreover, Strichartz~\cite{Str00} proved that $S_n f$ converges uniformly to $f$ for every $f\in C(\supp(\mu_{1/(2k)}))$. On the other hand, Dutkay, Han and Sun~\cite{DHS14} showed that there exists a continuous function on $\R$ whose Mock Fourier series w.r.t. the spectrum $17\Lambda_4$ of $\mu_{1/4}$ diverges at $0$. The question of almost everywhere convergence is more delicate. Strichartz in \cite{Str06} proved that for $k\ge 3$, for every $f\in L^1(\mu_{1/(2k)})$, $S_n f\to f$ $\mu_{1/(2k)}$-a.e.~$x\in \supp(\mu_{1/(2k)})$. We remark that this $L^1$ endpoint result has
no analogue in the classical theory, where Kolmogorov's example
shows that Fourier series of an $L^1(\mathbb{T})$ function may diverge almost
everywhere \cite{K26}, and the celebrated Carleson-Hunt theorem requires
$f\in L^p(\mathbb{T})$ with $p>1$ \cite{Car66,Hunt68}. Strichartz's a.e.\ argument succeeded
only for $k\ge 3$, leaving the case $k=2$ open in the same paper \cite[p.~341]{Str06}.

In this paper we settle this problem. Define the maximal operator associated with the family $\{S_n\}$  for $\mu_{1/4}$ w.r.t. the canonical spectrum $\Lambda_4$ by
 \begin{equation}\label{eq:maxmaxope}
S^\ast f(x) = \sup_n |S_nf(x)|
\end{equation}

Our main result is as follows:
\begin{thm}\label{thm:main}
The maximal operator $S^\ast$ defined in \eqref{eq:maxmaxope} is of weak type $(1,1)$, i.e., there exists $C>0$, such that for all $\alpha>0$,
 \begin{equation*}
  \mu_{1/4}\bigl(\{x:|S^\ast f(x)|>\alpha\}\bigr)
  \le \frac{C}{\alpha}\,\|f\|_{L^1(\mu_{1/4})}.
\end{equation*}
Consequently, for every $f\in L^1(\mu_{1/4})$,
\[
\lim_{n\to\infty} S_n f(x) = f(x)
\]
for $\mu_{1/4}$-almost every $x\in\supp(\mu_{1/4})$.
\end{thm}

The key novelty of our proof is to avoid the Hardy--Littlewood enlargement of the Dirichlet kernel used in the argument of Strichartz \cite{Str06}. Instead, we keep the average on its original cylinder and define positive terminal--shift operators $Q_{n,r}$ with coefficients $p=\frac{1}{4}+\frac{\pi}{6}<1$. A generalised aggregation lemma then sums the infinitely many weak-type estimates, yielding a weak-type $(1,1)$ bound for the maximal operator and proving almost everywhere convergence for \( k = 2 \). This removes the condition $\pi/(2k-1)<1$ for $k\ge 3$, which fails for $k=2$.

\textbf{Notation.} For simplicity, throughout the paper, we write $\mu=\mu_{1/4}$.

\section{Dirichlet kernel estimate and cylinder integration}

\subsection{Digit expansions and conditional expectations}

By \eqref{eq:expand}, each \(x\in\supp(\mu)\) has a unique expansion
\begin{equation}\label{eq:expansion}
 x = \sum_{j=1}^{\infty} \omega_j 4^{-j},
\end{equation}
where \(\omega_j\in\{-1,1\}\) for all \(j\ge 1\).  Consider the symbolic space $\Omega=\{-1,1\}^{\mathbb N}$ equipped with the Bernoulli measure $\mathbb P=(\frac12,\frac12)^{\mathbb N}$ and define \(\Phi:\Omega\to\supp(\mu)\) by
\[
\Phi(\omega)=\sum_{j=1}^{\infty}\omega_j 4^{-j},
\qquad \omega=\omega_1\omega_2\dots\in\Omega.
\]
Then \(\mu=\mathbb P\circ\Phi^{-1}\), and hence for every \(f\in L^1(\mu)\),
\[
\int_{\supp(\mu)} f(x)\,\mathrm{d}\mu(x)
=
\int_\Omega f(\Phi(\omega))\,\mathrm{d}\mathbb P(\omega).
\]

For $\eta=\eta_1\eta_2\cdots\in \{-1,1\}^\infty$ and $n\geq 1$, let \[
C_{n}(\eta)
=
\{\omega=\omega_1\omega_2\cdots\in\Omega:\omega_j=\eta_j,\ j=1,\dots,n\},
\]
be the level-\(n\) cylinder containing \(\eta\), which has \(\mathbb P\)-measure \(2^{-n}\).
We define \[\mathcal{F}_n=\sigma(\{\Phi(C_{n}(\eta)): \eta\in\{-1,1\}^\infty\})\] and $\mathcal{F}_\infty=\sigma(\bigcup_{n=1}^\infty\mathcal{F}_n)$. Then, one has $\mathcal{F}_1\subseteq \mathcal{F}_2\subseteq \cdots\subseteq \mathcal{F}_\infty$ and $\mathcal{F}_\infty$ is the set consisting of all $\mu$-measurable set.

For \(x=\Phi(\omega)\) with some \(\omega\in\Omega\), set
\[
I_n(x)=\Phi(C_n(\omega))\subseteq\supp(\mu),
\]
then, for every \(f\in L^1(\mu)\),
\[
 E_n(f)(x)
:= E[f\mid \mathcal{F}_n](x)=2^n\int_{I_n(x)} f(y)\,\mathrm{d}\mu(y)
.
\]
Hence,
\(
\{E_n(f),\mathcal F_n\}_{n\ge1}
\)
is a martingale on \((\mathcal F_\infty,\mu)\) since $\E[\E_{n+1}(f)\mid \mathcal F_n]=\E_{n}(f)$.  In particular, taking $f$ non-negative yields a non-negative martingale, to which Doob's maximal inequality will be applied in Section~6.

For \(\omega\in\Omega\) and a nonempty set \(S\subseteq\{1,\dots,n\}\), define \(\sigma_{n,S}\omega\in\Omega\) by
\[
(\sigma_{n,S}\omega)_j=
\begin{cases}
-\omega_j, & j\in S,\\
\phantom{-}\omega_j, & j\notin S.
\end{cases}
\]
The map \(\sigma_{n,S}:\Omega\to\Omega\) clearly preserves \(\mathbb P\), i.e., for every measurable set $A\subseteq\Omega$, $\mathbb P(\sigma_{n,S}^{-1}(A))=\mathbb P(A)$.
If \(x=\Phi(\omega)\in\supp(\mu)\), we denote
$I_{n,S}(x)=\Phi(C_n(\sigma_{n,S}\omega))$, then for $f\in L^1(\mu)$,
\begin{align}\label{eq:Aadd}
E_n(|f|)(\Phi(\sigma_{n,S}(\omega)))=2^n\int_{I_{n,S}(x)} |f(y)|\,\mathrm{d}\mu(y)
\end{align}

\subsection{The Dirichlet kernel and mismatch estimates}

Recall that the canonical spectrum of $\mu$ is given by
\[
\Lambda_{4}= \Biggl\{ \sum_{j=0}^{m-1}c_j 4^j : m\ge 1,\ c_j\in\{0,1\}\Biggr\}.
\]
For $n\ge 1$, let
\[
\Lambda_n:=\Lambda_{4}^{(n)}=\Biggl\{\sum_{j=0}^{n-1} c_j4^j : c_j\in\{0,1\}\Biggr\},
\]
and so $\#(\Lambda_n)=2^n$.  Here and in the sequel,  \(\#(A)\) denotes the cardinality of a set \(A\).

The {\it Dirichlet kernel} is defined by
$$D_n(x)=\sum_{\lambda\in\Lambda_{n}}e^{2\pi i\lambda x},$$
which is easily factorised as
\begin{equation}\label{eq:D}
D_n(x)=\prod_{j=0}^{n-1}\bigl(1+e^{2\pi i4^j x}\bigr)=:\prod_{j=0}^{n-1}F_j(x).
\end{equation}
This immediately yields that, for every $f\in L^1(\mu)$,
$$
S_n f(x)=\int_{\mathbb R} D_n(x-y)\,f(y)\,\mathrm d\mu(y).
$$

For any $x,y\in\supp(\mu)$. Write
$$
x = \sum_{k=1}^{\infty}\omega_k(x)\,4^{-k}
\quad\text{and}\quad
y = \sum_{k=1}^{\infty}\omega_k(y)\,4^{-k},
$$
with all $\omega_k(x),\omega_k(y)\in\{-1,1\}$. For each $j\geq0$,
\begin{equation}\label{eq:frac}
4^{\,j}(x-y)\equiv
\frac{\omega_{j+1}(x)-\omega_{j+1}(y)}{4}
+\sum_{l=2}^{\infty}\bigl(\omega_{j+l}(x)-\omega_{j+l}(y)\bigr)\,4^{-l}
\pmod 1.
\end{equation}
Define
$$
R_j(x,y)=\sum_{l=2}^{\infty}\bigl(\omega_{j+l}(x)-\omega_{j+l}(y)\bigr)\,4^{-l}.
$$
Since $|\omega_k(x)-\omega_k(y)|\le 2$ for all $k$, it follows that
\begin{equation}\label{eq:crude-R}
|R_j(x,y)|\le 2\sum_{l=2}^{\infty}4^{-l}=\frac{1}{6}.
\end{equation}

Now, for $n\geq1$, we define the following set, called the {\it mismatch set},
$$
M_n(x,y)=\{0\le j\le n-1:\omega_{j+1}(x)\neq\omega_{j+1}(y)\}.
$$
If $M_n(x,y)=\emptyset$, we only need the following trivial estimate:
\[|D_n(x-y)|\leq 2^n.\]
If $M_n(x,y)\neq\varnothing$, write $M_n(x,y)=\{g_j:1\leq j\leq m\}$ such that $g_1<g_2<\dots<g_m$.
In this case, a more delicate estimate is required. We estimate \(F_j(x-y)\) for \(0\le j\le n-1\) by distinguishing two cases.

Case A: $j\notin M_n(x,y)$. By the simple fact that $|1+e^{2\pi i\theta}|=2|\cos(\pi\theta)|$ for all $\theta\in\mathbb{R}$, we have
\begin{equation}\label{eq:matched}
|F_j(x-y)| = 2|\cos(\pi R_j)| \le 2
\end{equation}

Case B: $j=g_k\in M_n(x,y)$ for some $1\leq k\leq m$. Then
$$
\frac{\omega_{g_k+1}(x)-\omega_{g_k+1}(y)}{4}=\pm\frac12.
$$
By \eqref{eq:frac}, we have $4^{\,g_k}(x-y)\equiv \pm\frac12+R_{g_k}(x,y)\pmod 1$ and thus
\begin{equation}\label{eq:mismatched-sin}
|F_{g_k}(x-y)| = 2\bigl|\cos\bigl(\tfrac\pi2+\pi R_{g_k}(x,y)\bigr)\bigr|
=2|\sin(\pi R_{g_k}(x,y))|
\le 2\pi|R_{g_k}(x,y)|.
\end{equation}

It remains to estimate the bound on $R_{g_k}(x,y)$ for $1\leq k\le m$.  For $1\leq k\leq m-1$, by the definition of $M_n(x,y)$, we obtain
\begin{align*}
|R_{g_k}(x,y)|&\leq\sum_{i=2}^{\infty}|(\omega_{g_k+i}(x)-\omega_{g_k+i}(y))|4^{-i}\\
&=\sum_{l=g_{k+1}-g_k+1}^{\infty}
|\bigl(\omega_{g_k+l}(x)-\omega_{g_k+l}(y)\bigr)|\,4^{-l}\\
&\leq 2\cdot\frac{4^{-(g_{k+1}-g_k+1)}}{1-4^{-1}}\\
&= \frac{2}{3}\,4^{-(g_{k+1}-g_k)}.
\end{align*}

For $k=m$.  Since $g_m$ is the maximum element in $M_n(x,y)$,
one has
$$
\omega_{g_m+i}(x)=\omega_{g_m+i}(y) \qquad\text{for } 2\le i\le n-g_m.
$$
When $n-g_m>1$, the terms with $i=2,\dots,n-g_m$ vanish from $R_{g_m}(x,y)$, thus
\begin{align*}
|R_{g_m}(x,y)|&\leq\sum_{i=n-g_m+1}^{\infty}
|\bigl(\omega_{g_m+i}(x)-\omega_{g_m+i}(y)\bigr)|\,4^{-i}\\
&\le 2\sum_{i=n-g_m+1}^{\infty}4^{-i}
= \frac{2}{3}\,4^{g_m-n}.
\end{align*}
When $g_m=n-1$, we also have,
\begin{align*}
|R_{g_m}(x,y)|
&\leq\sum_{i=2}^{\infty}
|\bigl(\omega_{n-1+i}(x)-\omega_{n-1+i}(y)\bigr)|\,4^{-i}\\
&\leq 2\sum_{i=2}^{\infty}4^{-i}= \frac{2}{3}\,4^{g_m-n}.
\end{align*}
Denote $g_{m+1}=n$. Then $|R_{g_m}(x,y)|\le \frac{2}{3}\,4^{g_m-g_{m+1}}$.

Combining \eqref{eq:matched}
with the uniform bound $\frac{4\pi}{3}\,4^{\,g_k-g_{k+1}}$ of $|F_{g_k}(x-y)|$ for all $k=1,\dots,m$, we obtain
\begin{align*}
|D_n(x-y)|
&= \prod_{j=0}^{n-1} |F_j(x-y)| \notag\\
&\le 2^{\,n-m}\;
   \prod_{k=1}^{m}\Bigl(\frac{4\pi}{3}\;4^{\,g_k-g_{k+1}}\Bigr) \notag\\
&= 2^{\,n-m}\Bigl(\frac{4\pi}{3}\Bigr)^{\!m}\;
   4^{(g_1-n)}. 
\end{align*}
Thus we have proved the following lemma.

\begin{lem}\label{lem:mismatch}
Let $x,y\in\supp(\mu)$. For $n\geq1$,  denote $g=\min M_n(x,y)$ and $m=\#(M_n(x,y))$.   Then if $M_n(x,y)\neq\varnothing$
\begin{equation}\label{eq:strichartz-bound}
|D_n(x-y)|
\le 2^{\,n-m}\Bigl(\frac{4\pi}{3}\Bigr)^{\!m}4^{\,g-n}.
\end{equation}
If $M_n(x,y)=\varnothing$, $|D_n(x-y)|\le 2^n$.
\end{lem}

\begin{rem}
The idea of the above estimate for $\mu$ is from \cite{Str06} and the result matches the form of
\cite[eq.\,(2.11)]{Str06} after a change of digits.
\end{rem}

\subsection{Cylinder integration and control constants}

For a non-empty subset $S\subseteq\{0,\dots,n-1\}$, write $S'=\{j+1:j\in S\}(\subseteq\{1,\dots,n\})$. According to the definition in Subsection 2.1, for $x\in\supp(\mu)$, $I_{n,S'}(x)$ is the level-$n$ cylinder consisting of points whose first $n$ digit positions differ from those of $x$ exactly on $S'$. In particular, for any $y\in I_{n,S'}(x)$, the set $M_n(x,y)$ equals $S$. Integrating~\eqref{eq:strichartz-bound} over this sub-cylinder yields
\begin{align}
\int_{I_{n,S'}(x)}|D_n(x-y)|\,|f(y)|\,\mathrm d\mu(y)
&\le 2^{\,n-m}\Bigl(\frac{4\pi}{3}\Bigr)^{\!m}4^{\,g-n}
      \int_{I_{n,S'}(x)}|f(y)|\,\mathrm d\mu(y) \notag\\
&= \Bigl(\frac{2\pi}{3}\Bigr)^{\!m}4^{\,g-n}\,A_{n,S'}|f|(x),
\label{eq:cylinder-bound}
\end{align}
where $g,m$ are given in Lemma \ref{lem:mismatch} and $A_{n,S'}h(x)$ is defined as follows:
\[
A_{n,S'}h(x)=2^n\int_{I_{n,S}(x)} h(y)\,\mathrm{d}\mu(y).
\]

For $S=\varnothing$ we have
\begin{equation}\label{eq:diag}
\int_{I_{n,\varnothing}(x)}|D_n(x-y)|\,|f(y)|\,\mathrm d\mu(y)
\le E_n(|f|)(x).
\end{equation}

\section{Domination by terminal mismatch integral averages}
We reorganise the sum over level-$n$ sub-cylinders by
$r=n-g$, where $g=\min S$ as in Section~2.3.  For $1\le r\le n$ define
\begin{equation}\label{eq:Qnr}
Q_{n,r}f(x)=
\frac{1}{b(1+b)^{\,r-1}}
\sum_{\substack{S\subseteq\{n-r,\dots,n-1\}\\ \min S=n-r}}
b^{\#(S)}\,A_{n,S'}|f|(x),
\end{equation}
where $b=\frac{2\pi}{3}$ and $S'=\{j+1:j\in S\}$.  Then $Q_{n,r}$ is a convex combination of the operators $A_{n,S'}$, since
\begin{equation*}
\sum_{\substack{S\subseteq\{n-r,\dots,n-1\}\\ \min S=n-r}} b^{\#(S)}
= b\sum_{j=0}^{r-1}\binom{r-1}{j}b^{\,j}
=b(1+b)^{\,r-1}.
\end{equation*}

\begin{prop}\label{prop:domination}
Let $p=\frac{1}{4}+\frac{\pi}{6}$. For every $n\ge 1$ and $f\in L^1(\mu)$,
\[
|S_n f(x)|
\le E_{n}(|f|)(x) + \frac{b}{1+b}\sum_{r=1}^{n} p^{\,r}\,Q_{n,r}f(x).
\]
Consequently,
\[
\sup_{n\ge 1}|S_n f(x)|
\le  \sup_{n\ge 1} E_{n}(|f|)(x) + \frac{b}{1+b}\sum_{r=1}^{\infty} p^{\,r} M_r f(x),
\]
where
$
M_r f(x) = \sup_{n\geq r} Q_{n,r}f(x).
$
\end{prop}

\begin{proof}
Partition $\supp(\mu)$ into the $2^{\,n}$ cylinders of level~$n$. \(I_n(x)\) is the cylinder whose first $n$ digits agree with those of $x$, and other cylinders can be written as $I_{n,S'}(x)$ for non-empty $S\subseteq\{0,\dots,n-1\}$. According to \eqref{eq:cylinder-bound} and \eqref{eq:diag}, we have
\begin{align}\label{eq:leqleaeaq}
|S_n f(x)|
&\le \int_{\supp(\mu)}|D_n(x-y)|\,|f(y)|\,\mathrm d\mu(y) \nonumber \\
&= \int_{I_n(x)}|D_n(x-y)|\,|f(y)|\,\mathrm d\mu(y) \nonumber \\
&\quad + \sum_{r=1}^{n}
   \sum_{\substack{S\subseteq\{n-r,\dots,n-1\}\\ \min S=n-r}}
   \int_{I_{n,S'}(x)}|D_n(x-y)|\,|f(y)|\,\mathrm d\mu(y) \nonumber \\
&\le E_n(|f|)(x) + \sum_{r=1}^{n}
   \sum_{\substack{S\subseteq\{n-r,\dots,n-1\}\\ \min S=n-r}}
   4^{-r}b^{\#(S)}A_{n,S'}|f|(x) \nonumber \\
&= E_n(|f|)(x) + \frac{b}{1+b}\sum_{r=1}^{n} p^{\,r}\,Q_{n,r}f(x),
\end{align}
where the last equality holds by the simple fact that $4^{-r}b(1+b)^{\,r-1}=\frac{b}{1+b}\,p^{\,r}$. This proves the first assertion.

Taking the supremum over $n$ of both sides of \eqref{eq:leqleaeaq} and using that all summands
are non-negative yields the second claim. The proof is complete.
\end{proof}

\section{The terminal-digit shifted maximal lemma}

Given a non-empty subset $S\subseteq\{0,\dots,n-1\}$ and $S'=\{j+1:j\in S\}(\subseteq\{1,\dots,n\})$. For an integrable function $u$ on $\supp(\mu)$, recall that
$$
  A_{n,S'}u(x)=2^n\int_{I_{n,S'}(x)} u(y)\,\mathrm{d}\mu(y),
$$
where $I_{n,S'}(x)$ denotes
$\Phi(C_n(\sigma_{n,S'}\omega))$ for $x=\Phi(\omega)$, as introduced in Section 2.
Then set
\begin{equation}\label{eq:linear-Qnr}
  \widetilde Q_{n,r}u(x)=\frac{1}{b(1+b)^{\,r-1}}
  \sum_{\substack{S\subseteq\{n-r,\ldots,n-1\}\\
                   \min S=n-r}}
  b^{\#S}\,A_{n,S'}u(x),
  \qquad n\ge r.
\end{equation}
Therefore, $Q_{n,r}u(x)=\widetilde Q_{n,r}|u|(x)$.

The coefficients in~\eqref{eq:linear-Qnr} are non-negative and sum to
one, and each digit-flip map $\sigma_{n,S'}$ preserves $\mu$.
We now verify the $L^1$ and $L^\infty$ contraction properties
of $\widetilde Q_{n,r}$.  For any $u\in L^1(\mu)$, the
measure-preserving property of $\sigma_{n,S'}$ and Fubini's theorem give
\begin{align*}
  \|\widetilde A_{n,S'}u\|_{L^1(\mu)}
  &\leq \int_{\supp(\mu)} 2^n\int_{\supp(\mu)} |u(y)|\mathbf 1_{I_{n,S'}(x)}(y)\,\mathrm{d}\mu(y)\,\mathrm d\mu(x)\\
  &=2^n \int_{\operatorname{supp}(\mu)} |u(y)|
\left( \int_{\operatorname{supp}(\mu)} \mathbf 1_{I_{n,S'}(x)}(y)\,\mathrm d\mu(x) \right)\,\mathrm d\mu(y)\\
  &=\|u(x)\|_{L^1(\mu)}.
\end{align*}
Here and in the sequel, for a set $A$, $\mathbf{I}_A$ denotes its indicator function. 
  Since $\widetilde Q_{n,r}$ is a
convex combination of the operators $\widetilde A_{n,S'}$, we obtain
\begin{equation}\label{eq:L1-contraction}
  \|\widetilde Q_{n,r}u\|_{L^1(\mu)} \le \|u\|_{L^1(\mu)}.
\end{equation}
For the $L^\infty$ estimate, we have for every $x\in \supp(\mu)$
$$
  |\widetilde A_{n,S'}u(x)|
  \le 2^n\int_{I_{n,S'}(x)}|u(y)|\,\mathrm d\mu(y)
  \le \|u\|_{L^\infty(\mu)},
$$
 and the convex
combination yields
\begin{equation}\label{eq:Linfty-contraction}
  \|\widetilde Q_{n,r}u(x)\|_{L^\infty(\mu)} \le \|u\|_{L^\infty(\mu)}.
\end{equation}

The following theorem gives the required maximal estimate, with linear
growth in the complexity.

\begin{thm}\label{thm:shifted-maximal}
For every
$r\ge 1$, $\alpha>0$, and $f\in L^1(\mu)$,
$$
  \mu\bigl(\{x:M_r f(x)>\alpha\}\bigr)
  \le \frac{9r}{\alpha}\|f\|_{L^1(\mu)},
$$
where
$$
  M_r f(x)=\sup_{n\ge r}Q_{n,r}f(x).
$$
\end{thm}
\begin{proof}
The proof proceeds in two cases.

Case 1: $0<\alpha\le 3\|f\|_{L^1(\mu)}$. Then the trivial bound
\[
\mu(\{x:M_r f(x)>\alpha\})\le 1\le \frac{3}{\alpha}\|f\|_{L^1(\mu)}\le \frac{9r}{\alpha}\|f\|_{L^1(\mu)}
\]
holds since $\mu$ is a probability measure.

Case 2. $\alpha>3\|f\|_{L^1(\mu)}$. Put $h=|f|$.  We use the linear operators
$\widetilde Q_{n,r}$ defined above, so that
$Q_{n,r}f=\widetilde Q_{n,r}h$. For $x\in\supp(\mu)$ define
$$
  \tau(x)=\min\bigg\{n\ge 1:E_n(h)(x)>\frac{\alpha}{3}\bigg\},
$$
with $\tau(x)=\infty$ if no such $n$ exists. Let $\mathcal B$
be the family of maximal dyadic cylinders among
$\{I_{\tau(x)}(x):\tau(x)<\infty\}$ where $I_{\tau(x)}(x)$ is given in Section 2.1.

If $\mathcal B=\emptyset$, then by the definition of $\mathcal B$ we have $\tau(x)=\infty$ for $\mu$-a.e.~$x$. Recall that
\[
E_n(h)(x)=2^n\int_{I_n(x)} h(y)\,\mathrm{d}\mu(y)=E[h\mid \mathcal F_n](x),
\]
and $\{E_n(h),\mathcal F_n\}_{n\ge 1}$ is a martingale since $E[E_{n+1}(h)\mid \mathcal F_n](x)=E_n(h)(x)$ (see Section 2.1). Therefore, by Doob's martingale convergence theorem,
\begin{align}\label{eq4.4}
h(x)=\lim_{n\to\infty}E_n(h)(x)\le \frac{\alpha}{3}
\qquad \text{for } \mu\text{-a.e. }x.
\end{align}
According to \eqref{eq:Linfty-contraction},
\[
M_r f(x)=\sup_{n\ge r}|\widetilde Q_{n,r}h(x)|\le \|h\|_{L^\infty}\le \frac{\alpha}{3}
\qquad \text{for } \mu\text{-a.e. }x.
\]
Hence the set $\{x:M_r f(x)>\alpha\}$ has $\mu$-measure zero, and therefore
\[
\mu\bigl(\{x:M_r f(x)>\alpha\}\bigr)=0\le \frac{9r}{\alpha}\|f\|_{L^1(\mu)}.
\]

If $\mathcal B\neq\emptyset$, then by construction the cylinders in $\mathcal B$ are pairwise disjoint, and we set $B=\bigcup_{I\in\mathcal B}I$. For each $I\in\mathcal B$, by the
defining of $\mathcal B$, one has
$\frac{3}{\alpha}\int_I h(y)\,\mathrm d\mu(y)>\mu(I)$ and then
\begin{align}\label{eq4.5}
  \mu(B)=\sum_{I\in\mathcal B}\mu(I)\le\frac{3}{\alpha}\sum_{I\in\mathcal B}\int_I h(y)\,\mathrm d\mu(y)   \le \frac{3\|h\|_{L^1(\mu)}}{\alpha}.
\end{align}
Fix $I\in\mathcal B$, for any $x,y \in I$, we have $\tau(x)=\tau(y)$ from the the definition. So we can denote $\ell(I)=\tau(x)$ for some $x\in I$, and let $h_I=\frac1{\mu(I)}\int_I h(y)\,\mathrm d\mu(y)$.
When $\ell(I)=1$, one has
\begin{align*}
h_I=2\int_{I}h(y)\mathrm{d}\mu(y)\leq 2\|h\|_{L^1(\mu)}\leq\frac{2\alpha}{3}.
\end{align*}
When $\ell(I)>1$, from the definition of $\ell(I)$, we also have
\begin{align*}
h_I=E_{\ell(I)}(h)(x)\leq2E_{\ell(I)-1}(h)(x)\leq \frac{2\alpha}{3}\qquad \text{for some  $x\in I $.}
\end{align*}
Hence
\begin{align}\label{eq4.6}
h_I\leq \frac{2\alpha}{3} \qquad \text{for  $I\in\mathcal B$}.
\end{align}

Decompose
$$
  h(x) = g(x) + \sum_{I\in\mathcal B} b_I(x),
$$
where
$g(x) = h(x)\mathbf 1_{B^c}(x) + \sum_{I\in\mathcal B} h_I\mathbf 1_I(x)$ and $
  b_I(x) = \mathbf(h(x)-h_I\mathbf)\mathbf 1_I(x)$. Similar to \eqref{eq4.4},
\begin{align}\label{eq:addfeafea}
g(x)=h(x)=\lim_{n\rightarrow\infty}E_n (h)(x)\leq\frac{\alpha}{3} \qquad \text{for $\mu$ -a.e. $x\in B^c $},
\end{align}
Combing \eqref{eq:addfeafea} with \eqref{eq4.6}, one has
\begin{align*}
g(x)\leq\frac{2\alpha}{3}\qquad \text{for $\mu$ -a.e. $x\in \R $.}
\end{align*}
By~\eqref{eq:Linfty-contraction}, we obtain
\begin{align}\label{eq4.7}
\sup_{n\ge r}|\widetilde Q_{n,r}g(x)|
  \le \|g(x)\|_{L^\infty}
  \le \frac{2\alpha}{3} \qquad \text{for $\mu$ -a.e. $x\in \R $}.
\end{align}

Next, we estimate the value of $\widetilde Q_{n,r}b_I(x)$. Let $I\in\mathcal B$. If $n\leq\ell(I)$,  then for all $x\in \supp(\mu)$, $\widetilde A_{n,S'}b_I(x)=0$ because $I$ is contained in
a single level-$n$ cylinder and $\int_I b_I(y)\,\mathrm d\mu(y)=0$, hence $\widetilde Q_{n,r}b_I(x)=0$ for $x\in \supp(\mu)$;  If $n\ge \ell(I)+r$, by the definition of $\widetilde Q_{n,r}$, we know $S'\subseteq \{n-r+1,\ldots,n\}\subseteq\{\ell(I)+1,\ell(I)+2,\ldots\}$ which does not affect the first $\ell(I)$ digits. Since $A_{n,S'}b_I(x)$ is supported in $I$, it follows that $\widetilde Q_{n,r}b_I(x)$ is also supported in $I$, and therefore vanishes on $B^c$. So
\begin{align*}
\sum_{n\leq\ell(I), n\ge \ell(I)+r}\widetilde Q_{n,r}b_I(x)=0 \qquad \text{for $x\in B^c$}.
\end{align*}
Then, for $x\in B^c$,  we have
\begin{align*}
  \sup_{n\ge r}
  \Bigl|\widetilde Q_{n,r}\sum_{I\in\mathcal B}b_I(x)\Bigr|
  &\le  \sum_{I\in\mathcal B} \sup_{n\ge r} \Bigl|\widetilde Q_{n,r}b_I(x)\Bigr|  \\
  &\le     \sum_{I\in\mathcal B}
  \sum_{\substack{n\ge r\\
                  }}
  |\widetilde Q_{n,r}b_I(x)| \\
  &=
  \sum_{I\in\mathcal B}
  \sum_{\substack{n\ge r\\
                  \ell(I)\le n\le\ell(I)+r-1}}
  |\widetilde Q_{n,r}b_I(x)|,
\end{align*}
Together with \eqref{eq:L1-contraction}, we obtain
\begin{align}\label{eq4.8}
  \int_{B^c}
  \sup_{n\ge r}
  \Bigl|\widetilde Q_{n,r}\sum_{I\in\mathcal B}b_I(y)\Bigr|\,\mathrm d\mu(y)
  &\le
  \sum_{I\in\mathcal B}
  \sum_{\substack{n\ge r\nonumber \\
                  \ell(I)\le n\le\ell(I)+r-1}}
  \|\widetilde Q_{n,r}b_I(y)\|_{L^1(\mu)} \nonumber \\
  &\le
  r\sum_{I\in\mathcal B}\|b_I\|_{L^1(\mu)}
  \le 2r\|h\|_{L^1(\mu)}.
\end{align}
On the other hand, according to \eqref{eq4.7}, for $\mu$-a.e. $x\in B^c$,
$$
   Q_{n,r}h(x)
  \le |\widetilde Q_{n,r}g(x)|
     + \Biggl|\widetilde Q_{n,r}\sum_{I\in\mathcal B}b_I(x)\Biggr|
  \le \frac{2\alpha}{3}
     + \Biggl|\widetilde Q_{n,r}\sum_{I\in\mathcal B}b_I(x)\Biggr|.
$$
Then
$$
  \{x:M_r f(x)>\alpha\}
  \subseteq B
    \cup \Bigl\{x\in B^c:
      \sup_{n\ge r}\Bigl|\widetilde Q_{n,r}\sum b_I(x)\Bigr|>\frac{\alpha}{3}\Bigr\}.
$$
Combining with \eqref{eq4.5}, \eqref{eq4.8} and Markov's inequality,
\begin{align*}
  \mu\bigl(\{x:M_r f(x)>\alpha\}\bigr)
  &\le \mu(B)
     + \frac{3}{\alpha}
       \int_{B^c}
       \sup_{n\ge r}
       \Bigl|\widetilde Q_{n,r}\sum b_I\Bigr|\,\mathrm d\mu \\
  &\le \frac{3\|h\|_{L^1(\mu)}}{\alpha} + \frac{6r\|h\|_{L^1(\mu)}}{\alpha}\\
    &\le \frac{9r}{\alpha}\|f\|_{L^1(\mu)}.
\end{align*}

Hence, we finish the proof.
\end{proof}

\section{Proof of the main theorem}
By Theorem \ref{thm:shifted-maximal}, each operator \(M_r\) satisfies a weak-type \((1,1)\) bound
\[
  \|M_r f\|_{L^{1,\infty}}
  \le C(r+1)\|f\|_{L^1(\mu)}
\]
with constant \(C>0\) independent of $r$. Here,
\begin{equation*}
  \|M_rf\|_{L^{1,\infty}}
  :=\sup_{\alpha>0}\alpha\,
    \mu\bigl(\{x:|M_rf(x)|>\alpha\}\bigr).
\end{equation*}
Proposition~\ref{prop:domination} gives the following pointwise inequality:
\begin{equation}\label{eq:eqaddmaxi}
  \sup_{n\ge1}|S_n f(x)|
  \le Mf(x)
    + \frac{b}{1+b}\sum_{r\ge1} p^r M_r f(x),
\end{equation}
where $Mf(x)=\sup_{n\ge1} E_n(|f|)(x)$.

 For the first term on the right-hand side of \eqref{eq:eqaddmaxi},  as shown in
Section~2.1, $\{E_n(f),\mathcal F_n\}_{n\ge1}$ is a martingale on $(\mathcal F_\infty,\mu)$.  Applying Doob's maximal inequality gives for any $\alpha>0$,
\begin{equation}\label{eq:aeawewfawwefa}
\mu\bigl(\{x: M f(x)>\alpha\}\bigr)
\le \alpha^{-1}\|f\|_{L^1(\mu)},
\end{equation}
i.e., $M$ is of weak type $(1,1)$. Thus it remains to prove that the weighted series on the right-hand side of \eqref{eq:eqaddmaxi} is also of weak type $(1,1)$.

Since \(\|\cdot\|_{L^{1,\infty}}\) is only a quasi-norm, it satisfies only the quasi-triangle inequality
\[
\|u+v\|_{L^{1,\infty}}
\le 2\bigl(\|u\|_{L^{1,\infty}}+\|v\|_{L^{1,\infty}}\bigr).
\]
This makes the direct summation of weak-type bounds fail.
Indeed, applying the quasi-triangle inequality to the first \(N\) terms of the weighted series gives
\begin{align}\label{eq:adddqwefwqew}
\Bigl\|\sum_{r=1}^N p^r M_r f\Bigr\|_{L^{1,\infty}}
&\le 2^{N-1}\sum_{r=1}^N p^r \|M_r f\|_{L^{1,\infty}}\\
&\le C\,2^{N-1}\|f\|_{L^1(\mu)}\sum_{r=1}^N p^r (r+1).
\end{align}
Since \(0<p<1\), the series \(\sum_{r\ge1} p^r (r+1)\) converges to a positive finite number \(L\). Thus the right-hand side of \eqref{eq:adddqwefwqew} is bounded by \(C L\,2^{N-1}\|f\|_{L^1(\mu)}\), which tends to infinity as \(N\to\infty\). Therefore we avoid the quasi-norm and instead estimate the level sets of the weighted sum directly. The following lemma supplies the required summation bound.

\begin{lem}\label{lem:sum-weak}
  Let \((X,\nu)\) be a finite measure space and let
  \(T_r\), \(r\ge 1\), be operators on \(L^{1}(\nu)\) taking non-negative
  values. Assume that each \(T_r\) satisfies a weak-type \((1,1)\)
  estimate
  \begin{equation}\label{eq:weak-hyp}
    \nu\bigl(\{x\in X: T_r f(x)>\alpha\}\bigr)
    \le C_r\,\alpha^{-1}\|f\|_{L^{1}(\nu)}
  \end{equation}
  for every \(\alpha>0\) and \(f\in L^1(\nu)\). If
  \(\{a_r\}_{r\ge 1}\) is a sequence of positive numbers such that
  \(\sum_{r\ge 1}\sqrt{a_r C_r}<\infty\), then the weighted sum
  \(\sum_{r\ge 1}a_r T_r\) also satisfies a weak-type \((1,1)\)
  estimate:
  \begin{equation}\label{eq:weak-sum}
    \nu\Bigl(\Bigl\{x\in X:
      \sum_{r\ge 1} a_r T_r f(x) > \alpha\Bigr\}\Bigr)
    \le \frac{1}{\alpha}\Bigl(\sum_{r\ge 1}\sqrt{a_r C_r}\Bigr)^{\!2}
       \|f\|_{L^{1}(\nu)}
  \end{equation}
  for every \(\alpha>0\) and \(f\in L^{1}(\nu)\).
\end{lem}
\begin{proof}
  Let $\{c_r\}_{r\ge 1}$ be any sequence of strictly positive numbers
  with \[\sum_{r\ge 1}c_r=1.\]  Suppose that
  $\sum_{r\ge 1} a_r T_r f(x)>\alpha$ at some point $x\in X$.
  If we had $T_r f(x)\le\alpha c_r/a_r$ for every $r\ge 1$, then multiplying by $a_r$ and summing over $r$ would give
  \[
    \sum_{r\ge 1}a_r T_r f(x)
    \le\sum_{r\ge 1}\alpha c_r=\alpha,
  \]
  which contradicts the assumption. Hence there exists a $r_0$ for
  which $T_{r_0} f(x)>\alpha c_{r_0}/a_{r_0}$, and we obtain
  $$
    \Bigl\{x\in X: \sum_{r\ge 1} a_r T_r f(x) > \alpha\Bigr\}
    \subseteq
    \bigcup_{r\ge 1}
    \Bigl\{x\in X: T_r f(x) > \frac{\alpha c_r}{a_r}\Bigr\}.
  $$
  Applying the subadditivity of $\nu$, and the
  weak-type hypothesis~\eqref{eq:weak-hyp},
  we obtain
  \begin{align}\label{eq:subsub}
    \nu\Bigl(\Bigl\{x: \sum_{r\ge 1} a_r T_r f(x) > \alpha\Bigr\}\Bigr)
    &\le \sum_{r\ge 1}
       \nu\Bigl(\Bigl\{x: T_r f(x) > \frac{\alpha c_r}{a_r}\Bigr\}\Bigr) \notag \\[4pt]
    &\le \sum_{r\ge 1}
       C_r\cdot\frac{a_r}{\alpha c_r}\,\|f\|_{L^{1}(\nu)} \notag \\[4pt]
    &= \frac{\|f\|_{L^{1}(\nu)}}{\alpha}
       \sum_{r\ge 1} \frac{a_r C_r}{c_r}.
  \end{align}
  Since the left-hand side of \eqref{eq:subsub} is independent of $\{c_r\}$, it follows that \begin{equation}\label{eq:min-form}
    \nu\Bigl(\Bigl\{x: \sum_{r\ge 1} a_r T_r f(x) > \alpha\Bigr\}\Bigr)
    \le \frac{\|f\|_{L^{1}(\nu)}}{\alpha}\,
        \inf_{\substack{c_r>0\\ \sum_rc_r=1}}
        \sum_{r\ge 1}\frac{a_r C_r}{c_r}.
  \end{equation}

By the Cauchy-Schwarz
  inequality,
  \[
    \Bigl(\sum_{r\ge 1}\sqrt{a_r C_r}\Bigr)^{\!2}\le \Bigl(\sum_{r\ge 1}\frac{a_r C_r}{c_r}\Bigr)
        \Bigl(\sum_{r\ge 1}c_r\Bigr)= \sum_{r\ge 1}\frac{a_r C_r}{c_r},
 \]
  and equality is attained when
  $c_r = \sqrt{a_r C_r}\,/\!\sum_{s\ge 1}\sqrt{a_s C_s}$.
  Consequently,
  $$
    \inf_{\substack{c_r>0\\ \sum_r c_r=1}}
      \sum_{r\ge 1}\frac{a_r C_r}{c_r}
    = \Bigl(\sum_{r\ge 1}\sqrt{a_r C_r}\Bigr)^{\!2},
  $$
  which, inserted into \eqref{eq:min-form}, yields~\eqref{eq:weak-sum}.
\end{proof}


\begin{proof}[Proof of Theorem~\ref{thm:main}]
By Proposition~\ref{prop:domination}, for every \(f\in L^1(\mu)\),
\[
S^\ast f(x)
\le M f(x) + \frac{b}{1+b}\sum_{r\ge1} p^r M_r f(x).
\]
The maximal operator \(M\) is of weak type \((1,1)\) by \eqref{eq:aeawewfawwefa}. Moreover, Theorem~\ref{thm:shifted-maximal} gives
\[
\mu\bigl(\{x:M_r f(x)>\alpha\}\bigr)
\le C(r+1)\alpha^{-1}\|f\|_{L^1(\mu)}
\]
for all \(r\ge1\). Set \(a_r=\frac{b}{1+b}p^r\). Since \(p<1\),
\[
\sum_{r\ge1}\sqrt{a_r C(r+1)}
=\sqrt{\frac{bC}{1+b}}\sum_{r\ge1}\sqrt{r+1}\,p^{r/2}<\infty.
\]
Thus Lemma~\ref{lem:sum-weak} applied to \(T_r=M_r\) on \((\supp\mu,\mu)\) shows that the weighted sum \(\sum_{r\ge1} a_r M_r f\) is of weak type \((1,1)\). Consequently, \(S^\ast\) is of weak type \((1,1)\).

It remains to prove almost-everywhere convergence from the weak-type
$(1,1)$ estimate. We follow the standard procedure from \cite{St70}. Define for $h\in L^1(\mu)$
$$
\Omega h(x) = \limsup_{n\to\infty} |S_n h(x) - h(x)|,
\qquad x\in\supp\mu.
$$
For any $h_1,h_2\in L^1(\mu)$, the triangle inequality gives
$$
\Omega(h_1+h_2)(x) \le \Omega h_1(x) + \Omega h_2(x)\quad\text{and}\quad
\Omega h_1(x) \le S^\ast h_1(x) + |h_2(x)|.
$$
If $f\in C(\supp\mu)$, then Strichartz~\cite{Str00} proved that
$S_n f$ converges uniformly to $f$ on $\supp\mu$; hence
$\Omega f(x)=0$ for every $x\in\supp\mu$.

Fix $\varepsilon>0$.
Since $C(\supp\mu)$ is dense in $L^1(\mu)$, there exists
$h\in C(\supp\mu)$ such that $\|f-h\|_{L^1(\mu)} < \varepsilon^2$.
Then
\begin{align*}
\Omega f(x)& \le \Omega(f-h)(x) + \Omega h(x)
           = \Omega(f-h)(x)\\
           &\le S^\ast(f-h)(x) + |f(x)-h(x)|.
\end{align*}
For any $\alpha>0$, the weak-type estimate for $S^\ast$ and Markov's
inequality give
\begin{align*}
\mu\bigl(\{x: \Omega f(x) > \alpha\}\bigr)
&\le \mu\bigl(\{x:
      S^\ast(f-h)(x) > \alpha/2\}\bigr) \\
&\qquad + \mu\bigl(\{x:
      |f(x)-h(x)| > \alpha/2\}\bigr)\\[2pt]
&\le \frac{\widetilde C\,\varepsilon^2}{\alpha/2}
   + \frac{\varepsilon^2}{\alpha/2}
 = \frac{2(\widetilde C+1)}{\alpha}\,\varepsilon^2 .
\end{align*}
Since $\varepsilon>0$ is arbitrary, $\mu\bigl(\{x: \Omega f(x) > \alpha\}\bigr)=0$ for
every $\alpha>0$. Hence $\Omega f = 0$ $\mu$-almost everywhere on
$\supp\mu$, which is precisely
$\lim_{n\to\infty}S_n f(x) = f(x)$ for $\mu$-almost every
$x\in\supp\mu$.
\end{proof}

\section*{Acknowledgements}
The first author was supported by Hunan Provincial Natural Science Foundation of China (2026JJ90109) and the Scientific Research Fund of Hunan Provincial Education Department (25A0086, 24B0107). The second author was supported by  the National Natural Science Foundations of China (12271534, 12301105) and the University Research Project of Guangzhou Education Bureau (2024312332).

\end{document}